\documentclass[12pt]{amsart}
\usepackage{amsmath,amssymb}
\usepackage{color}

\theoremstyle{plain}
\newtheorem{theorem}{Theorem}
\newtheorem{corollary}{Corollary}

\newtheorem{lemma}{Lemma}

\newtheorem{remark}{Remark}  
\newtheorem*{ack}{Acknowledgment}

\renewcommand{\div}{\operatorname{div}}

\newcommand{\nref}[1]{(\ref{#1})}
\newcommand{\ncite}[1]{{\bf \cite{#1}}}

\theoremstyle{definition}
\newtheorem{definition}{Definition}

\theoremstyle{remark}

 \numberwithin{definition}{section}
 \numberwithin{proposition}{section}
 \numberwithin{remark}{section}
\numberwithin{theorem}{section}
\numberwithin{lemma}{section}
\numberwithin{corollary}{section}

\numberwithin{equation}{section}
\begin{document} 
 
\title{  Rearrangements in Carnot groups} 
\author
{Juan J. Manfredi}
\address{Department of Mathematics \\
University of Pittsburgh\\
Pittsburgh, PA 15260, USA}
\email{manfredi@pitt.edu}
\author{Virginia N. Vera de Serio}
\address{Facultad de Ciencias Econ\'omicas\\
	Universidad Nacional de Cuyo\\
	5500 Mendoza, Argentina}
\email{virginia.vera@fce.uncu.edu.ar}
 
\thanks{First author supported in part by NSF grant
DMS-9970687. \color{blue}Second author by SECTyP-UNCuyo, Argentina, Res. 3853/16-R.} \normalcolor
\date{December 23, 2000}

\begin{abstract}
In this paper we extend the  notion of rearrangement of nonnegative 
functions to the setting of Carnot groups. 
We define  rearrangement 
 with respect to a given family of anisotropic balls $B_{r}$, or
 equivalently with respect to a gauge $\|x\|$, 
and prove basic regularity properties of this construction. If $u$ 
is a bounded nonnegative real function
with compact support, we  denote by $u^{\star}$ its 
rearrangement. Then, the radial function $u^{\star}$ is of  bounded 
 variation. In addition, if $u$ is continuous  then $u^{\star}$
is continuous, and  if $u$ belongs to the 
horizontal Sobolev space $W^{1,p}_{\textrm{h}}$, then
$ \frac{D_{\mathrm{h}}u^{\star}(x)}{|D_{\mathrm{h}}(\|x\|)|}$ is in $L^p $.
Moreover, we found a 
generalization of the inequality of P\'olya and Szeg\"o
$$\int \frac{|D_{\textrm{h}}u^{\star}|^{p}}{
|D_{\mathrm{h}}(\|x\|)|^{p}}\, dx\le
\textrm{C} \,\,
\int |D_{\textrm{h}}u|^{p}\, dx,$$
where $p\ge 1$.

\end{abstract}
\maketitle
\section{Introduction}
Let $u\colon \mathbb{R}^{n}\mapsto\mathbb{R}$ be a non-negative
measurable
real function with compact support. The rearrangement of $u$ is the radial function 
$u^{\star}$ that has
the same distribution function as $u$ with respect to the Lebesgue
measure $\mathcal{L}^{n}$. 
That is, for every $\lambda>0$ we have
$$\mathcal{L}^{n}(\{x\colon u^{\star}(x)>\lambda\})=
\mathcal{L}^{n}(\{x\colon u(x)>\lambda\}).
$$
In particular, for any non-negative Borel measurable
real function $\phi$ we have
$$\int_{\mathbb{R}^{n}}\phi( u^{\star}(x))\,d\mathcal{L}^{n}(x)=
\int_{\mathbb{R}^{n}}\phi(u(x))\,d\mathcal{L}^{n}(x).$$\par
P\'olya and Szeg\"o proved in \ncite{PS} that if 
$u\in W^{1,p}(\mathbb{R}^{n})$, where $p\ge 1$, then so is
$u^{\star}$ and we have the inequality
$$\int_{\mathbb{R}^{n}}|Du^{\star}(x)|^{p}\,d\mathcal{L}^{n}(x)\le
\int_{\mathbb{R}^{n}}|Du(x)|^{p}\,d\mathcal{L}^{n}(x).$$
If the Lebesgue measure $\mathcal{L}^{n}$ is replaced by
another measure  we get a different rearrangement.
The  motivation for this article originated from  a result
of Schulz and Vera de Serio \ncite{SV} concerning rearrangements in
$\mathbb{R}^{2}$ relative to an absolutely continuous measure  
with
respect to $\mathcal{L}^{2}$ with a density $\rho$. The rearrangement
$u^{\star}_{\rho }$ of a non-negative function
$u$ is determined by the condition 
$$\int_{\{u^{\star}_{\rho}>\lambda\}}\rho(x)\, d\mathcal{L}^{2}(x)
=\int_{\{u >\lambda\}}\rho(x)\,d\mathcal{L}^{2}(x)$$
for every $\lambda\ge 0$. One of the main results in  \ncite{SV} states
that if $\log \rho$ is a nonnegative 
sub-harmonic function in $\mathbb{R}^{2}$,
then for every non-negative $u\in W^{1,2}(\mathbb{R}^{2})$
the rearrangement
$u^{\star}_{\rho}\in W^{1,2}(\mathbb{R}^{2})$ and we have
the inequality
$$\int_{\mathbb{R}^{2}}|Du^{\star}_{\rho}(x)|^2
\, d \mathcal{L}^{2}(x)\le
\int_{\mathbb{R}^{2}}|Du(x)|^{2}\,d\mathcal{L}^2(x).$$
\par
In this paper we define rearrangements in general spaces that
include   Carnot-Carath\'eodory spaces, and prove inequalities
of   P\'olya-Szeg\"o type in the case of Carnot groups. \color{blue} For a comprehensive recent symmetrization reference, see the book
\cite{B}. 
\normalcolor
 
\section{ Real Variable Structures for Rearrangements}
We are given a family $\{B_{r}\}$ of non-empty bounded open
sets, \lq\lq balls centered at $0$\rq\rq, in $\mathbb{R}^{n}$ indexed by $r>0$ 
satisfying the following
conditions:
\begin{align}
&	r<s   \implies B_{r}\subset B_{s}, \label{primera}\\
&	\bigcap_{r>0}B_{r}=\{0\},\label{segunda}\\
&	\bigcup_{r>0}B_{r}=\mathbb{R}^{n},\quad\textup{and} \\
&	\bigcup_{0<r<s}B_{r}=B_{s}.\label{union}
\end{align}
We also set $B_{0}=\emptyset$. For $x\in\mathbb{R}^{n}$ we define
the  gauge
$$\|x\|=\inf\{r>0: x \in B_{r}\}$$ and assume that
\begin{equation}
	x\mapsto \|x\|\qquad \textup{is  a continuous function}.
	\end{equation}
It follows easily that 
\begin{equation}\label{ball}
	B_{r}=\{x\colon \|x\|<r\}.
	\end{equation}\par
We are also given a non-negative Borel measure $\mu$ in $\mathbb{R}^{n}$
such that the 
volume function $$ V(r)=\mu(B_{r})$$ satisfies the following
properties:
\begin{align}
	&V(0)=\textstyle{\lim_{r\to 0^{+}}V(r)=0}, \label{vuno}\\
	&V(\infty)=\mu(\mathbb{R}^{n}),\label{vdos}\\
	&V\colon[0,\infty]\mapsto[0,\mu(\mathbb{R}^{n})]\quad
	\textup{is an absolutely continuous bijection}.\label{ultima}
\end{align}
Let $u\colon\mathbb{R}^{n}\mapsto[0, \infty]$ be a non-negative
$\mu$-measurable function \color{blue} with compact support \normalcolor. For each $t\ge 0$ define
\begin{align}
	&E_{u}(t)=\{x\in\mathbb{R}^{n}\mid u(x)>t\},\\
	&  \nu_{u}(t)=\mu(E_{u}(t)),\quad\textup{and} \\
	& \tilde{\nu}_{u}(r)=\sup\{t\colon \nu_{u}(t)>V(r)	\}.
\end{align}
We follow the convention  $\sup\emptyset=0$. 
We are ready for our general definition of rearrangement.
\begin{definition}
	Given a family of non-empty bounded open sets $\{B_{r}\}_{r>0}$ and
	a Borel measure $\mu$ such that properties \nref{primera}
	through \nref{ultima} hold,   the   rearrangement
	of a   $\mu$-measurable function $u:\mathbb{R}^{n}\mapsto[0,\infty]$
	is the \lq\lq radial\rq\rq\ function 
	$u^{\star}:\mathbb{R}^{n}\mapsto[0,\infty]$ defined by
	$$u^{\star}(x)=\tilde{\nu}_{u}(\|x\|).$$
\end{definition}

The following lemma is elementary.
\begin{lemma}
	The function $\tilde{\nu}_{u}$ is finite, non-increasing and
	continuous from the right on $(0,\infty)$. Moreover, we have
	$$\tilde{\nu}_{u}(0)=\tilde{\nu}_{u}(0^{+})=\textup{ess}\sup u.
	$$
	\end{lemma}
 Observe that the equivalence
 $$\nu_{u}(t)>V(r)\iff \mu(E_{u}(t))>\mu(B_{r})$$
always holds.
\begin{corollary}\label{unoseis}
	If the gauge $x\mapsto\|x\|$ is differentiable $\mu$-a.~e., then
	the gradient	$Du^{\star}(x)$ exists $\mu$-a.~e. 
	and satisfies
	$$Du^{\star}(x)=\tilde{\nu}_{u}'(\|x\|)\cdot D(\|x\|).$$
\end{corollary}
\begin{proof}It is enough to observe that the absolute 
	continuity of the
	volume function gives   $\mu(\{x\colon \|x\|\in A\})=0$
	whenever $A$ is a set of measure zero in $\mathbb{R}$.
	\end{proof}
	
\begin{lemma} For every $t\ge 0$ we have 
	 $$\mu(E_{u^{\star}}(t))=\mu(E_{u}(t)).$$
	 Therefore $u$ and $u^{\star}$ have the same distribution
	 function with respect to the measure $\mu$.
\end{lemma}
	\begin{proof} Let us start observing that
		the level set $E_{u^{\star}}(t)$ is the ball $B_{r}$ where
	$r=V^{-1}(\nu_{u}(t))$. Indeed, we have
	$$E_{u^{\star}}(t)=\{u^{\star}>t\}=
	\bigcup_{s>t}B_{V^{-1}(\nu_{u}(s))}=B_{V^{-1}(\nu_{u}(t))},$$
	because $V^{-1}\circ \nu_{u}$ is right continuous 
	and property \nref{union}. The lemma follows  from
	the following chain of equalities
	$$\mu(E_{u^{\star}}(t))=\mu(B_{r})=V(r)=V(V^{-1}(\nu_{u}(t)))=
	\nu_{u}(t)=\mu(E_{u}(t)).
	$$
\end{proof}
\begin{corollary}\label{menem}
	For any non-negative Borel measurable real function $\phi$ we have
	$$\int_{\mathbb{R}^{n}}\phi( u^{\star}(x))\,d\mu(x)=
\int_{\mathbb{R}^{n}}\phi(u(x))\,d\mu(x).$$
\end{corollary}
\begin{lemma} If $u$ is continuous and has compact support then
	$\nu_{u}$ is strictly decreasing on the interval
	$[0,\textup{ess}\sup u]$ and $V^{-1}\circ \nu_{u}$ is a
	right inverse of $\tilde{\nu}_{u}$.
\end{lemma}
	\begin{proof} The proof is identical to the proof 
		of Lemma 1.4.1 and Lemma 1.5.1 in \ncite{SV}, since
		only properties \nref{primera} through \nref{ultima}
		are used.
		\end{proof}
\begin{theorem}\label{ucontinua}
	If $u$ is continuous with compact support so
	is $u^{\star}$.
\end{theorem}
\begin{proof} Once again, the proof is identical to the proof 
		of Theorem 1.4.4 in \ncite{SV}, since
		only properties \nref{primera} through \nref{ultima}
		are used.
		\end{proof}
\begin{lemma}\label{vhypo}
	If $u$ is continuous with compact support then
	\begin{itemize}
		\item[(i)] $\tilde{\nu}_{u}$ is continuous and,
		\item[(ii)] if  $\nu_{u}'(r)\not= 0$ for a.~e. 
		$r\in [0,\textup{ess}\sup u]$ 
		then $\tilde{\nu}_{u}$ is absolutely continuous.
		\end{itemize}
\end{lemma}
\begin{proof}
	The continuity of $\tilde{\nu}_{u}$ follows from the
	continuity of $u^{\star}$. The argument for (ii) is
	the same as in the proof of Proposition 1.5.2 in
	\ncite{SV} using the absolute continuity of $V$.
\end{proof}
\section{  Rearrangements in Carnot groups.}
 Consider a collection
of $m$ smooth vector fields in $\mathbb{R}^{n}$
$$\{X_{1},X_{2},\ldots,X_{m}\}$$ satisfying H\"ormander's condition
$$\textup{Rank Lie}[X_{1},X_{2},\ldots,X_{m}](x)=n$$
at every $x\in\mathbb{R}^{n}$. 
We will also assume that the horizontal tangent space
$$T_{\mathrm{h}}(x)=\textup{Linear span}[X_{1},X_{2},\ldots,X_{m}](x)$$
has dimension $m\le n$
for all $x\in\mathbb{R}^{n}$.\par
A piecewise smooth 
curve $t\mapsto\gamma(t)\in\mathbb{R}^{n}$ is horizontal
if its tangent vector
$\gamma'(t)$ is in   $T_{\mathrm{h}}(\gamma(t))$.
The Carnot-Carath\'{e}odory distance between 
 the points $p$ and $q$ is defined  as follows: 
$$
d_{CC}(p,q)=\inf \{ \textup{length}(\gamma)\mid \gamma\in \Gamma \}
$$
where the set $ \Gamma $
is the set of all horizontal curves $ \gamma $ such that 
$ \gamma (0) = p$ and  $ \gamma (1) = q $. To measure the length
of a curve we use the metric in $T_{\mathrm{h}}(x)$
determined by requiring that  the vector fields
$\{X_{1},X_{2},\ldots,X_{m}\}$ form an orthonormal basis. We
can always extend this metric to a full Riemannian metric
in $\mathbb{R}^{n}$ so that  
its volume element is the Lebesgue measure $\mathcal{L}^{n}$.\par

By Chow's theorem (see, for example,
\ncite{BR}) any two points
 can be connected by a horizontal curve, which
makes $d_{CC}$  a metric on $ \mathbb{R}^{n}$.
A Carnot-Carath\'{e}odory ball of radius $r$ centered at a
point $p_0$ is given by
$$B(p_0,r)=\{p\in  \mathbb{R}^{n}: d_{CC}(p,p_0) < r\}.$$

Observe that properties \nref{primera} to \nref{ball} always
hold in an arbitrary metric space if $B_{r}=B(x_{0},r)$ is
the family of balls centered at some fixed point $x_{0}$, where
$0$ in property \nref{segunda} is replaced by $x_{0}$.\par

Given a Borel measure $\mu$ property \nref{vdos} always holds
and so does \nref{vuno} if $\mu$ is non-atomic. Property 
\nref{ultima} follows easily if $\mu$ is absolutely continuous
with respect to $\mathcal{L}^{n}$.  \par

>From now on we will consider the case of
a Carnot group $\mathcal{G}$ of dimension $n$ and
homogenous dimension $Q$ as defined, for example, in \ncite{FS}.
The vector fields 
$\{X_{1},X_{2},\ldots,X_{m}\}$ are  left-invariant so that
we think of them as elements in the Lie algebra $\mathfrak{g}$.
The Haar measure of the group is $\mathcal{L}^{n}$ and we
have a family of group homomorphisms $\delta_{r}$ indexed
by $r>0$, called
dilations, satisfying
$$\delta_{r}\circ \delta_{s}=\delta_{rs}.$$
The volume of a ball is given by
$$V(B(p_{0},R)=\textup{constant}\cdot R^{Q}.$$
In order to determine a real variable structure for
rearrangements we have to single out a gauge $x\to\|x\|$
and set $B_{R}=\{x\mid \|x\|< R\}$. There are many
choices of   gauges which are smooth away from the
origin, see \ncite{FS}. A gauge that is usually
non-smooth but natural in our setting is  the Carnot
gauge
$$\|x\|_{C}=d_{CC}(x,0).$$
 
We occasionally identify $\mathcal{G}$ with the 
underlying space $\mathbb{R}^{n}$.

\begin{theorem}\label{ccsimetrico}
	A Carnot group $\mathcal{G}$  endowed with the  
	Carnot gauge $\|x\|_{C}$,
	or with a smooth gauge $x\mapsto\|x\|$ 
	together with the Lebesgue
	measure $\mathcal{L}^{n}$ forms a real variable 
	rearrangement structure. That is, properties
	\nref{primera} through \nref{ultima} hold.
\end{theorem}
In particular Theorem  \nref{ucontinua} applies to
a Carnot group endowed with an arbitrary gauge.
\par

The horizontal gradient of a function $u\colon\mathcal{G}\mapsto
\mathbb{R}$ is the projection of the full gradient onto the
horizontal tangent space
$$D_{h}u=(X_{1}u)X_{1}+(X_{2}u)X_{2}+\ldots+(X_{m}u)X_{m}.$$
For $p\ge 1$ the horizontal Sobolev space  is defined by
$$W_{\textrm{h}}^{1,p}(
\mathcal{G})=\left\{
u\in L^{p}(\mathcal{G})\mid D_{\textrm{h}}u\in L^{p}(\mathcal{G})
\right\}.$$
Endowed with the norm
$$\|u\|_{W_{\textrm{h}}^{1,p}(\mathcal{G})}=
\|u\|_{L^{p}(\mathcal{G})}+
\|D_{\textrm{h}}u\|_{L^{p}(\mathcal{G})},$$
the class $W_{\textrm{h}}^{1,p}(\mathcal{G})$ is a Banach space (see
\ncite{GN} and \ncite{L}).\par

The horizontal divergence $\div_{\textrm{h}}(F)$
of a horizontal vector field $F$ 
($F(x)=\sum_{i=1}^{m}F^{i}(x)X_{i}(x)$)
is defined by requiring
that for every  compactly supported   smooth function $\phi$ the
equality 
$$\int_{\mathcal{G}} \phi \div_{\textrm{h}}(F) d\mathcal{L}^{n}
=-\int_{\mathcal{G}}\langle D_{\textrm{h}}\phi, F \rangle
d\mathcal{L}^{n}
$$
holds.  \par
Next, we recall the definition of horizontal bounded variation
from \ncite{GN}. We say that $u\in BV_{\mathrm{h}}(\Omega)$
if
$$\|u\|_{BV_{\mathrm{h}}(\Omega)}=
\sup\left\{ \int_{\Omega}u \div_{\mathrm{h}}F\,
d\mathcal{L}^n
\right\}<\infty.$$
where the supremum is taken among all 
$F\in C_{0}^{\infty}(\Omega,
\mathfrak{g})$ such that $\sum_{i=1}^{m}|F^{i}(x)|^{2}\le 1$. 
If the function $u$ is smooth,
the horizontal bounded variation is just the $L^{1}$-norm of
the length of the horizontal gradient
$$\|u\|_{BV_{\mathrm{h}}(\Omega)}=
\int_{\Omega}|D_{\mathrm{h}}u|\,d\mathcal{L}^{n}.$$\par
A measurable set $E\subset \mathcal{G}$ has finite horizontal 
perimeter relative to a domain $\Omega\subset\mathcal{G}$ 
if $\chi_{E}\in BV_{\mathrm{h}}(\Omega)$ in which case
we write
$$\mathcal{P}_{\mathrm{h}}(E, \Omega)=\|\chi_{E}\|_{ BV_{\mathrm{h}}
(\Omega)}.$$
We shall denote $\mathcal{P}_{\mathrm{h}}(E, \mathcal{G})$ simply
by $\mathcal{P}_{\mathrm{h}}(E)$.\par
Using the anisotropic dilations, it is easy to see that
\begin{equation}\label{perihomo}
	\mathcal{P}_{\mathrm{h}}(B_{R})= R^{Q-1}
	\mathcal{P}_{\mathrm{h}}(B_{1}).
\end{equation}

\begin{theorem} Suppose that $\mathcal{G}$ is a Carnot group
endowed with
a gauge so that the unit ball $B_{1}$ is regular enough to have
finite horizontal perimeter.
Let $u\in L^{\infty}(\mathcal{G})$ be a nonnegative
function with compact support.
Then $u^{\star}\in BV_{\mathrm{h}}(\mathcal{G})$.  
\end{theorem}
\begin{remark}The finiteness of the horizontal perimeter
	of a ball  
certainly holds for   smooth gauges   and also 
for the Carnot gauge 
$\|x\|_{C}$ in a general Carnot group. See Remark 4.3 in
\ncite{MSC}.
	\end{remark}
\begin{proof} We will use integration in polar coordinates.
	For $r>0$ set
	$$\phi(r)=\int_{B_{r}} \div_{\mathrm{h}} F(y) \, d\mathcal{L}^{n}(y).$$ It follows
	from  Proposition 1.15 in \ncite{FS}, that there exists a 
	Radon measure $\sigma$ on $\partial B_{1}$ such that
	$$\phi'(r)=\int_{\partial B_{r}}\div_{\mathrm{h}} F(y) \,
	d\sigma_{r}(y),$$
	where $d\sigma_{r}$ is the image of the measure $d\sigma$ under
	the dilation $x\mapsto \delta_{r}(x)$.\par
	 
		Let $F\in C_{0}^{\infty}(\mathcal{G},
\mathfrak{g})$ be a test field satisfying $|F|\le 1$. We have
\begin{align*}
\int_{\mathcal{G}}u^{\star}(x) \div_{\mathrm{h}}F(x)\,
d\mathcal{L}^{n}(x)&=
\int_{0}^{\infty}\int_{\partial B_{r}}
u^{\star}(y)\div_{\mathrm{h}}F(y)\, d\sigma_{r}(y)\, dr\\
&=\int_{0}^{\infty}\tilde\nu_{u}(r)\left(
\int_{\partial B_{r}}\div_{\mathrm{h}}F(y)\, d\sigma_{r}(y)
\right)\, dr\\
&=\int_{0}^{\infty}\tilde\nu_{u}(r)\phi'(r), dr\\
&=-\int_{0}^{\infty}\phi(r) d\tilde\nu_{u}(r)
 \end{align*}
Observe next that $\phi(R)\le \mathcal{P}_{\mathrm{h}}(B_{R})$.
Using \nref{perihomo} we get
$$\phi(R)\le \mathcal{P}_{\mathrm{h}}(B_{1})\cdot R^{Q-1}.$$ Since
$-d\tilde\nu_{u}$ is a positive measure we see that
$$\int_{\mathcal{G}}u^{\star}(x) \div_{\mathrm{h}}F(x)\, d\mathcal{L}^{n}x
\le  -\mathcal{P}_{\mathrm{h}}(B_{1})\int_{0}^{\infty} r^{Q-1}
\, d\tilde\nu_{u}(r)<\infty.$$
 Therefore $u^{\star}\in BV_{\mathrm{h}}(\mathcal{G})$. 
 
\end{proof}

A basic result that we shall use several times is the  
isoperimetric inequality for horizontal perimeters. (See
Garofalo and Nhieu \ncite{GN}  and    Franchi, Gallot and
Wheeden \ncite{FGW}).
Recall that $Q$ is the homogeneous
dimension of our Carnot group. For every set $E$ with finite
horizontal perimeter $\mathcal{P}_{\mathrm{h}}(E)< \infty$ we have
\begin{equation}\label{iso}
	\left(\mathcal{L}^{n}(E)\right)^{\frac{Q-1}{Q}}\le 
	C_{\textrm{iso}}\,\, \mathcal{P}_{\mathrm{h}}(E),
	\end{equation}
	where $C_{\textrm{iso}}$ is a constant independent of the set $E$.
Inequality \eqref{iso} follows from Theorem 1.18 in \ncite{GN} by
taking the domain $\Omega$ in this theorem to be a metric ball
of radius $R$ and letting  $R\to\infty$. \par
Garofalo and Nhieu \ncite{GN} and Franchi, Gallot
and Wheeden \ncite{FGW} also extended Federer's classical
co-area formula to the subelliptic setting.\par
\textbf{Horizontal Co-area Formula:} Let $\Omega\subset\mathcal{G}$
be a domain and let $u\in BV_{\mathrm{h}}(\Omega)$. Then, for 
a.~e. $t\in\mathbb{R}$, the set 
$$E_{u}(t)=\{x\in\mathcal{G}\mid u(x)>t\}$$ has finite horizontal
perimeter relative to $\Omega$ and the co-area formula 
\begin{equation}\label{coarea}
	\|u\|_{BV_{\mathrm{h}}(\Omega)}=\int_{\mathbb{R}}
\mathcal{P}_{\mathrm{h}}(E_{u}(t), \Omega)\, dt.
\end{equation}
holds.
Conversely, for  $u\in L^{1}(\Omega)$, if for a.~e.
 $t\in\mathbb{R}$ the set 
$E_{u}(t)$ has finite horizontal
perimeter relative to $\Omega$, and  
$$\int_{\mathbb{R}}
\mathcal{P}_{\mathrm{h}}(E_{u}(t), \Omega)\, dt<\infty,$$
then $u\in BV_{\mathrm{h}}(\Omega)$ and we have \eqref{coarea}.
\par
For a function $u\in BV_{\mathrm{h}}(\mathcal{G})$ recall
the variation measure $\|D_{\mathrm{h}}u\|$ defined by
$$\|D_{\mathrm{h}}u\|(U)=
\sup\left\{\int_{\mathcal{G}} u \div_{\mathrm{h}} F\,
d\mathcal{L}^{n}
\right\},
$$
where $U$ is an open set in $\mathcal{G}$ and the supremum
is taken with respect to 
$F\in C_{0}^{\infty}(U, \mathfrak{g})$  such
that $\sum_{i=1}^{m}|F^{i}(x)|^{2}\le 1$. With this notation, the horizontal perimeter of a
set $E$ relative to a domain $\Omega$
is just $\|D_{\mathrm{h}}\chi_{E}\|(\Omega)$.
 We can also write \eqref{coarea} as follows
 \begin{equation}\label{coarea2}
\|	 D_{\mathrm{h}} u\|(\Omega)=\int_{\mathbb{R}}
\mathcal{P}_{\mathrm{h}}(E_{u}(t), \Omega)\, dt.
	 \end{equation}
	 Hence, for any nonnegative Borel measurable $g$ we have
	 \begin{equation}\label{coarea3}
		 \int_{\mathcal{G}} g \, d \|D_{\mathrm{h}} u\|=
		\int_{\mathbb{R}}\int_{\mathcal{G}}
		g\, d\mathcal{P}_{\mathrm{h}}(E_{u}(t))\,dt.
		 \end{equation}
\begin{lemma}\label{pag9}
	Consider a function $u$ in
	the horizontal Sobolev space
	$W^{1,1}_{\mathrm{h}}(\mathcal{G})$.
	Given a number  $ t<\|u\|_{\infty}$ and $s>t$ we have
	$$\mathcal{L}^{n}\left( u^{-1}(t,s)\right)>0.$$
	\end{lemma}
	\begin{proof}
Suppose that $\mathcal{L}^{n}\left( u^{-1}(t,s)\right)=0$.
Let $g$ be a smooth function with compact support bounded
by $1$. Write
\begin{align*}\int_{\mathcal{G}} X_{i}u \cdot g \,
	d\mathcal{L}^{n}
&	=
\int_{u\le t}X_{i}u\cdot g \, d\mathcal{L}^{n}+
\int_{u\ge s}X_{i}u\cdot  g \, d\mathcal{L}^{n}\\
&=-\int_{\mathcal{G}}X_{i}(t-u)^{+}\cdot g\,d\mathcal{L}^{n}
+\int_{\mathcal{G}}X_{i}(u-s)^{+}\cdot g\,d\mathcal{L}^{n}\\
&=\int_{u\le t}(t-u)\cdot X_{i}g\,d\mathcal{L}^{n}+
\int_{u\ge s}(s-u)\cdot X_{i}g\,d\mathcal{L}^{n},
\end{align*}
where we have used the lattice properties of 
$W^{1,1}_{\mathrm{h}}(\mathcal{G})$ (Lemma 3.5 in \ncite{GN})
and integration by  parts.\par
On the other hand we also have
\begin{align*}\int_{\mathcal{G}}X_{i}u\cdot g\,d\mathcal{L}^{n}&=
	\int_{\mathcal{G}}X_{i}(u-t)\cdot g\,d\mathcal{L}^{n}\\
	&=-\int_{\mathcal{G}}(u-t)\cdot X_{i} g\,d\mathcal{L}^{n}\\
	&=\int_{u\le t}(t-u)\cdot X_{i}g\, d\mathcal{L}^{n}
	+\int_{u\ge s}(t-u)\cdot X_{i}g \,d\mathcal{L}^{n}.
\end{align*}
We conclude that
\begin{equation}\label{efe}
	\int_{u\ge s}X_{i} g\,d\mathcal{L}^{n}=0.
	\end{equation}
If we call $E=\{u\ge s\}$, it follows from \eqref{efe} that
$\mathcal{P}_{\mathrm{h}}(E)=0$.  Since sets of horizontal perimeter
zero have $\mathcal{L}^{n}$ measure zero as it follows from
the horizontal isoperimetric inequality \eqref{iso},
we deduce that $u(x)\le t$ for a.~e. $x\in\mathcal{G}$ 
contradicting the hypothesis $t<\|u\|_{\infty}$.
	\end{proof}
\begin{theorem}
	If $u\in W_{\mathrm{h}}^{1,1}(\mathcal{G})\cap L^{\infty}$ is a nonnegative
	function with compact support, then
	$u^{\star}\in W_{\mathrm{h}}^{1,1}(\mathcal{G})$.
	\end{theorem}
	\begin{proof}Once we have Lemma \ref{pag9}, the isoperimetric
		inequality \eqref{iso} and the 
		coarea formula \eqref{coarea3}, the proof is identical to
		the proof of Theorem 1.6.7 in \ncite{SV}.
		\end{proof}
\section{Energy inequality for $p=1$}
We begin with a   lemma showing a quasi-monotonicity
property of the horizontal perimeter under rearrangements.
 
 \begin{lemma}\label{keylemma} There exists a constant
	 $C_{per}\ge 1$ such that for all sets $E\subset\mathcal{G}$
	 we have
	   \begin{equation}
	\mathcal{P}_{\mathrm{h}}(E^{\star})
	 \le C_{per}
	\mathcal{P}_{\mathrm{h}}(E),
	\end{equation}
	where $E^{\star}$ is the ball $B_{R}$ satisfying 
	$\mathcal{L}^{n}(B_{R})=\mathcal{L}^{n}(E)$.
	 \end{lemma}
\begin{proof} Observe that if $B_{R}$ is a ball, then
	$$\mathcal{L}^{n}(B_{R})=R^{Q}\mathcal{L}^{n}(B_{1})$$
	and
	$$\mathcal{P}_{\mathrm{h}}(B_{R})=R^{Q-1}
	\mathcal{P}_{\mathrm{h}}(B_{1}).$$
Therefore, we have the following equality for balls 
$$ \left(\mathcal{L}^{n}(B_{R})
\right)^{\frac{Q-1}{Q}}=C_{0} \mathcal{P}_{\mathrm{h}}(B_{R}),$$
where we have set
\begin{equation}\label{czero}
	C_{0}=\frac{\left(\mathcal{L}^{n}(B_{1})
\right)^{\frac{Q-1}{Q}}}{\mathcal{P}_{\mathrm{h}}(B_{1}) }.
\end{equation}
We now combine \nref{czero} with the isoperimetric inequality
\nref{iso} as follows:
$$	\mathcal{P}_{\mathrm{h}}(E^{\star})=
\frac{1}{C_{0}} \left(\mathcal{L}^{n}(E^{\star})
\right)^{\frac{Q-1}{Q}}= \frac{1}{C_{0}} \left(\mathcal{L}^{n}(E)
\right)^{\frac{Q-1}{Q}} \le \frac{C_{iso}}{C_{0}} 
\mathcal{P}_{\mathrm{h}}(E).
$$
We conclude that 
\begin{equation}\label{cpercisoc0}
	C_{per}\le \frac{C_{iso}}{C_{0}}.
\end{equation}
\end{proof}

Note that if  $u\in BV_{\mathrm{h}}(\mathcal{G})$, 
we have
\begin{equation}\label{bvarrange}
	\mathcal{P}_{\mathrm{h}}\left(E_{u^{\star}}(t)
	\right)\le C_{per}
	\mathcal{P}_{\mathrm{h}}\left(E_{u}(t)
	\right).
	\end{equation}

This follows from the fact   that the level set $E_{u^{\star}}(t)$ is the
	ball $B_{R}$ where $R=V^{-1}(\nu_{u}(t))$ and the previous Lemma.\par
 
\begin{theorem}\label{flaco}  
	For all nonnegative
	$u\in W_{\mathrm{h}}^{1,1}(\mathcal{G})$ with compact support
	 we have the inequality
	\begin{equation}
		\int_{\mathcal{G}} |D_{\textrm{h}}u^{\star}(x)|\, d\mathcal{L}^{n}(x)\le
\textrm{C}_{per}\,\,
\int_{\mathcal{G}} |D_{\textrm{h}}u(x)|\, d\mathcal{L}^{n}(x)
	\end{equation}
In particular, it follows that 		
	$u^{\star}\in W_{\mathrm{h}}^{1,1}(\mathcal{G})$. 
\end{theorem}
\begin{proof}
	Using the co-area formula twice, we get:
\begin{align*}
 	 \int_{\mathcal{G}} |D_{\mathrm{h}}u^{\star}(x)|\, 
	 d\mathcal{L}^{n}(x) 
		&=\int_{0}^{\infty} \left(
		\int_{\mathcal{G}}
		d\mathcal{P}_{\mathrm{h}}(E_{u^{\star}}(t))\right)\,dt\\
		&=\int_{0}^{\infty} 
		\mathcal{P}_{\mathrm{h}}(E_{u^{\star}}(t))\,dt\\
		&\le \textrm{C}_{per} \int_{0}^{\infty} 
		\mathcal{P}_{\mathrm{h}}(E_{u}(t))\,dt\\
		&= \textrm{C}_{per} \int_{\mathcal{G}} |D_{\mathrm{h}}u(x)| \,
		 d\mathcal{L}^{n}(x)  .
\end{align*}
 \end{proof}
 
\section{Energy inequality for $p\ge1$}

In this section we need to assume that the mapping
$x\mapsto \|x\|$ is differentiable a.~e. This is certainly
the case for smooth gauges and also for the Carnot
gauge. In fact, Monti and Serra Cassano \ncite{MSC}
have recently established that 
  the Carnot gauge $\|x\|_{C}$ is differentiable a.~e. and 
 satisfies
\begin{equation}
	\label{delaruaoff}\left| D_{\mathrm{h}}\left(\|x\|_{C}\right)
 \right| =1 \text{  for  a.~e. }x\in \mathcal{G}.\end{equation}
 \par
The key step to obtain the  
rearrangement energy inequality for $p\ge 1$    is an
integrability  property of
$$\frac{1}{|D_{\mathrm{h}}(\|x\|)|}\cdot$$

\begin{lemma} \label{integral2}
	For an arbitrary a.~e. differentiable gauge in a Carnot group we have
	 $$\int_{\mathcal{G}}
	\frac{1}{\left|D_{\mathrm{h}}(\|x\|)\right| }\,
	d\mathcal{P}_{\mathrm{h}}(B_{R})\le\,\,  \,\,
	R^{Q-1}\sigma(B_{1}),$$
	where $\sigma$ is the Radon measure supported 
	on $\partial B_{1}$ that is used in   integration
in polar coordinates. 
\end{lemma}

\begin{proof}  
Let us begin by observing that both sides of the inequality
are homogeneous of degree $Q-1$. Therefore, it is enough to
prove the lemma when $R=1$. We will write $B$ for $B_{1}$.  \par
Let $U$ be an open set and compute
\begin{align*}
	\int_{U}  d\mathcal{P}_{\mathrm{h}}&(B)(x)
	=\mathcal{P}_{\mathrm{h}}(B)(U)\\
	&=\sup\left\{\int_{{B}} \div_{\textrm{h}}
	F(x)\,d\mathcal{L}^{n}(x)\mid F\in C^{\infty}_{0}(U,\mathfrak{g}),
	 |F|\le 1
	\right\}\\
	&=\sup\left\{\int_{\partial B}\langle F(x), D_{\mathrm{h}}(\|x\|)
	\rangle\,  d\sigma(x) \mid F\in C^{\infty}_{0}(U,\mathfrak{g}),
	 |F|\le 1
	\right\},
\end{align*}
where the last equality follows from
\begin{equation}\label{formula}
	\int_{{B}} \div_{\textrm{h}}
	F(x)\,d\mathcal{L}^{n}(x)=\int_{\partial B}\langle F(x), D_{\mathrm{h}}(\|x\|)
	\rangle\,  d\sigma(x).\end{equation}
To prove this formula, consider the continuous function
$\phi_{\epsilon}(r)$ which takes the value $1$ for $r<1-\epsilon$,
  vanishes for $r>1+\epsilon$ and is linear otherwise.
>From the definition of horizontal divergence we get
\begin{align*}
	\int_{B_{1+\epsilon}}\phi_{\epsilon}(\|x\|)\div_{\mathrm{h}}&F(x)
\, d\mathcal{L}^{n}(x) =
-\int_{B_{1+\epsilon}}\langle D_{\mathrm{h}}(\phi_{\epsilon}(\|x\|)),
F(x)
\rangle\, d\mathcal{L}^{n}(x)\\
&=-\int_{B_{1+\epsilon}} \phi_{\epsilon}'(\|x\|)
\langle D_{\mathrm{h}}( \|x\|),
F(x)
\rangle\, d\mathcal{L}^{n}(x)\\
&=\frac{1}{2 \epsilon}\int_{B_{1+\epsilon}\setminus B_{1-\epsilon}}
\langle D_{\mathrm{h}}( \|x\|),
F(x)
\rangle\, d\mathcal{L}^{n}(x)\\
&=\frac{1}{2 \epsilon}
\int_{1-\epsilon}^{1+\epsilon}\int_{\partial B}
\langle D_{\mathrm{h}}( \|x\|),
F(\delta_{t}(x))
\rangle\, d\sigma(x)t^{Q-1}dt,
\end{align*}
where we  have 
used the fact that $D_{\mathrm{h}}( \|x\|)$ is  homogenous
of degree zero. Letting $\epsilon\to 0$ we obtain
\nref{formula}. \par
Thus, we have
\begin{equation*}
	\int_{U}  d\mathcal{P}_{\mathrm{h}}(B)(x)\le 
	\int_{\partial B \cap U}|D_{\mathrm{h}}(\|x\|)|\, d\sigma(x).
\end{equation*}
Since this is an inequality between two Radon measures, we conclude
that for $f$ nonnegative and Borel measurable
$$\int_{\mathcal{G}}f(x)\,d\mathcal{P}_{\mathrm{h}}(B)(x)
\le 	\int_{\partial B  }f(x) |D_{\mathrm{h}}(\|x\|)|\, d\sigma(x).
$$
The lemma follows by applying this formula to
$$f(x)=\frac{1}{|D_{\mathrm{h}}(\|x\|)|}\cdot$$
\end{proof}
Next, we need to discuss a technical point. 
It follows from 
 Corollary \ref{unoseis} that 
	$$D_{\mathrm{h}}u^{\star}(x)=\tilde{\nu}_{u}'(\|x\|)\cdot
	D_{\mathrm{h}}(\|x\|).$$
Since $|\tilde{\nu}_{u}'|$ is measurable with respect to
the $\sigma$-algebra
generated by $\tilde{\nu}_{u}$, there exists a
Borel measurable function
$\Psi\colon \mathbb{R}^+\cup \{0\} \to\mathbb{R}^{+}\cup\{0\}$ such that
$$|\tilde{\nu}_{u}'|(s)=\Psi(\tilde{\nu}_{u}(s)).$$
Therefore, using the equality $\tilde{\nu}_{u}(\|x\|)=
u^{\star}(x)$ we can write
\begin{equation}\label{modulodh}
	|D_{\mathrm{h}}u^{\star}(x)|=\Psi(u^{\star}(x)) \cdot
	|D_{\mathrm{h}}(\|x\|)|.
	\end{equation}
Observe that the factor $\Psi(u^{\star}(x))$ is radial but this
is not, in general, the case of the second factor
$|D_{\mathrm{h}}(\|x\|)|$. Nevertheless,
with the choice of the Carnot gauge this factor is identically $1$
 and $\left| D_{\mathrm{h}} u^{\star}(x)
 \right|$ is indeed a radial function. \par
 One could possibly  think that  
$|D_{\mathrm{h}}u^{\star}(x)|$  
is measurable with respect to
the $\sigma$-algebra
generated by $u^{\star}$
so that we had $|D_{\mathrm{h}}u^{\star}(x)|=\Phi(u^{\star}(x))$
for some Borel function $\Phi$. This is  actually  the case
for the Carnot gauge,  but it is not for
other gauges for which $|D_{\mathrm{h}}(\|x\|)|$ is not
radial. This is why we need
Lemma \nref{integral2}.
\par
\begin{theorem}\label{gordo} Let $\mathcal{G}$ be a Carnot group
	endowed with an a.~e. differentiable gauge.
	Let $u\in W^{1,p}_{\mathrm{h}}(\mathcal{G})$ be a nonnegative
	function with compact support and $p\ge 1$. There
	exists a positive constant $C_{sym}$ such that  we have
	the inequality
\begin{equation}\label{desigualdad}
	\int_{\mathcal{G}}\frac{
	\left| D_{\mathrm{h}} u^{\star}(x)\right|^{p}}
	{ \left| D_{\mathrm{h}}(\|x\|)  \right|^{p}}\,d\mathcal{L}^{n}(x)\le 
(C_{sym})^{p} \int_{\mathcal{G}}\left| D_{\mathrm{h}}
u(x)\right|^{p}\,d\mathcal{L}^{n}(x).
	\end{equation}
In fact, we may take    
$$C_{sym}= \frac{\sigma(B_{1})}{\mathcal{P}_{\mathrm{h}}(B_{1})} 
C_{per}.$$
	\end{theorem}
 \begin{proof}
	 Let $\Psi_{k}=\min\{k, \Psi\}$ be the truncation of
	 $\Psi$ at level $k$. By  the
	 coarea formula \nref{coarea3}
	 we get:
	 \begin{align*}
		 \int_{\mathcal{G}}\Psi_{k}^{p}(u^{\star}(x))\,
		 d\mathcal{L}^{n}(x) &\le
		 \int_{\mathcal{G}}\frac{\Psi_{k}^{p-1}(u^{\star}(x))}
		 {\left| D_{\mathrm{h}}(\|x\|) \right|}
		 \left| D_{\mathrm{h}} u^{\star}(x)\right| \,d\mathcal{L}^{n}(x)\\
		 &=\int_{0}^{\infty} \int_{\mathcal{G}}
		 \frac{\Psi_{k}^{p-1}(u^{\star}(x))}
		 {\left| D_{\mathrm{h}}(\|x\|)  \right|}
		 d\mathcal{P}_{\mathrm{h}}\left(E_{u}^{\star}(t)
		 \right)\,dt\\
		 &=\int_{0}^{\infty}\Psi_{k}^{p-1}(t)\left[
		 \int_{\mathcal{G}}\frac{1}{ \left| D_{\mathrm{h}}
		 (\|x\|)  \right|}
		 d\mathcal{P}_{\mathrm{h}}\left(E_{u}^{\star}(t)
		 \right) \right]\,dt\\
		 \end{align*}
At this time we use the key lemma \nref{integral2} together
with lemma \nref{keylemma}  and
another application of the coarea formula \nref{coarea3} 
and corollary \nref{menem} to get:
\begin{align*}
	\int_{\mathcal{G}}\Psi_{k}^{p}&(u^{\star}(x))\,d\mathcal{L}^{n}(x) \le
C_{sym } \int_{0}^{\infty}\Psi_{k}^{p-1}(t)
	\left[
		 \int_{\mathcal{G}}  
		 d\mathcal{P}_{\mathrm{h}}\left(E_{u}(t)
		 \right) \right]\,dt\\
		 &= C_{sym }  \int_{0}^{\infty}\int_{\mathcal{G}}
		 \Psi^{p-1}_{k}(u(x))d\mathcal{P}_{\mathrm{h}}\left(E_{u}(t)
		 \right)\,dt\\
		 &= C_{sym }  \int_{\mathcal{G}} \Psi^{p-1}_{k}(u(x))
		 \left|D_{\mathrm{h}}u(x) \right|\,d\mathcal{L}^{n}(x)\\
		 &\le C_{sym }  \left(\int_{\mathcal{G}}\Psi^{p}_{k}(u(x))\,
		 d\mathcal{L}^{n}(x)
		 \right)^{\frac{p-1}{p}} \left(\int_{\mathcal{G}}
		 \left|D_{\mathrm{h}}u(x) \right|^{p}\,d\mathcal{L}^{n}(x)
		 \right)^{\frac{1}{p}}\\
		 &= C_{sym } \left(\int_{\mathcal{G}}
		 \Psi^{p}_{k}(u^{\star}(x))\,d\mathcal{L}^{n}(x)
		 \right)^{\frac{p-1}{p}} \left(\int_{\mathcal{G}}
		 \left|D_{\mathrm{h}}u(x) \right|^{p}\,d\mathcal{L}^{n}(x)
		 \right)^{\frac{1}{p}}.
	\end{align*}
Hence, we obtain
$$	\left(\int_{\mathcal{G}}\Psi_{k}^{p}(u^{\star}(x))
\,d\mathcal{L}^{n}(x)
\right)^{\frac{1}{p}}\le C_{sym } 
\left(\int_{\mathcal{G}}
		 \left|D_{\mathrm{h}}u(x) \right|^{p}\,d\mathcal{L}^{n}(x)
		 \right)^{\frac{1}{p}},$$
		 letting $k\to\infty$ and using \nref{modulodh} we end the proof.
	 \end{proof}
	
For the Carnot gauge, we can prove a more traditional version
of the energy inequality. In this case
$|D_{\mathrm{h}}u^{\star}(x)|$ is radial and from
\nref{modulodh} it can be written
in the form
  
\begin{equation}\label{delarua}
	|D_{\mathrm{h}}u^{\star}(x)|=\Psi(u^{\star}(x)).
	\end{equation}
 
\begin{theorem}\label{gordo2}Let $\mathcal{G}$ be a Carnot
	group endowed with the Carnot gauge.
	Let $u\in W^{1,p}_{\mathrm{h}}(\mathcal{G})$ be a nonnegative
	function with compact support  and $p\ge 1$. Then, we have
	the inequality
\begin{equation}\label{desigualdad2}
	\int_{\mathcal{G}}
	\left| D_{\mathrm{h}} u^{\star}(x)\right|^{p}
	 \,d\mathcal{L}^{n}(x)\le 
(C_{per})^{p} \int_{\mathcal{G}}\left| D_{\mathrm{h}} u(x)
\right|^{p}\,d\mathcal{L}^{n}(x).
	\end{equation}
	In particular, it follows that 
	$u^{\star}\in W^{1,p}_{\mathrm{h}}(\mathcal{G})$.
	\end{theorem}
 \begin{proof}
	 Let $\Psi_{k}=\min\{k, \Psi\}$ be the truncation of
	 $\Psi$ at level $k$. By  the
	 coarea formula \nref{coarea3}
	 we get:
	 \begin{align*}
		 \int_{\mathcal{G}}\Psi_{k}^{p}(u^{\star}(x))\,
		 d\mathcal{L}^{n}(x) &\le
		 \int_{\mathcal{G}}\Psi_{k}^{p-1}(u^{\star}(x))
		 \left| D_{\mathrm{h}} u^{\star}(x)\right| \,d\mathcal{L}^{n}(x)\\
		 &=\int_{0}^{\infty} \int_{\mathcal{G}}
		 \Psi_{k}^{p-1}(u^{\star}(x))
		  d\mathcal{P}_{\mathrm{h}}\left(E_{u}^{\star}(t)
		 \right)\,dt\\
		 &=\int_{0}^{\infty}\Psi_{k}^{p-1}(t)\left[
		 \int_{\mathcal{G}} 
		 d\mathcal{P}_{\mathrm{h}}\left(E_{u}^{\star}(t)
		 \right) \right]\,dt\\
		 \end{align*}
 Next, we use    lemma \nref{keylemma} together
with another application of the coarea formula \nref{coarea3} 
and   \nref{delarua} to repeat the arguments of the second
part of the proof of Theorem \nref{gordo}  to end 
		 the proof.
	 \end{proof}

 \begin{ack}Part of this work was done while the first author 
 visited the Universidad Nacional de  Cuyo as a 
 FOMEC visiting professor during the Summers of 1998 and 2000.
  He wishes to express his appreciation 
 for the kind hospitality and the nice working atmosphere.\par
 We did not submit this manuscript for publication because we were trying to proof the sharp version of Theorem 5; that is $C_{per}=1$, which to the best of our knowledge remains open as of today. Nevertheless, the preprint was circulated among the community. We received several requests from colleagues, and it has been quoted several times in the literature.  \color{blue}See for example
 the book \cite{CDPT}, the dissertation \cite{La},  and the papers \cite{T}, \cite{LL}, \cite{LLT}, \cite{CLLY}, \cite{F1}, \cite{LLZ},  and \cite{F2}. \normalcolor  In addition, the main results have not been published elsewhere. This is why we decided to submit this manuscript to publication. We are very thankful to Acta Mathematica Sinica, English Series for considering this manuscript.\par
\color{blue} We are thankful to the referees for the careful reading and thoughtful suggestions. 
\normalcolor
\end{ack}

\end{document}